\documentclass[12pt]{amsart}

\usepackage{amsmath,bm,amssymb,amsthm,textcomp}
\usepackage{enumitem}
\usepackage{mathtools}
\usepackage[numbers,sort&compress]{natbib}
\usepackage{hyperref}
\usepackage{xcolor}

\DeclareMathOperator{\id}{id}

\theoremstyle{plain}
\newtheorem{theorem}{Theorem}[section]
\newtheorem{lemma}[theorem]{Lemma}
\newtheorem{proposition}[theorem]{Proposition}
\newtheorem{corollary}[theorem]{Corollary}

\theoremstyle{definition}

\theoremstyle{remark}
\newtheorem{remark}[theorem]{Remark}

\numberwithin{equation}{section}

\allowdisplaybreaks

\begin{document}

\title[Topology for Engel expansions: evaluation and coding]{A topology for Engel expansions: evaluation and digit coding maps}

\author{Min Woong Ahn}
\address{Department of Mathematics Education, Dankook University, 152, Jukjeon-ro, Suji-gu, Yongin-si, Gyeonggi-do, 16890, Republic of Korea}
\email{minwoong@silla.ac.kr}

\date{\today}


\maketitle

\begin{abstract}
We develop a topological framework for Engel expansions that treats both directions of the correspondence between points of $(0,1]$ and nondecreasing digit sequences. We endow the sequence space with the product topology to study the evaluation map, and we fix a nonterminating digit algorithm to study the digit coding map. We also record the correspondence between cylinder sets and fundamental intervals, and give an application to Baire category results for functions of the digits.
\end{abstract}

\section{Introduction}

Every $x\in(0,1]$ admits an Engel expansion
\begin{equation}\label{eq:intro-Engel}
x=\frac{1}{d_1(x)}+\frac{1}{d_1(x)d_2(x)}+\frac{1}{d_1(x)d_2(x)d_3(x)}+\dotsb,
\end{equation}
with integer digits $2\le d_1(x)\le d_2(x)\le\dotsb$. Set, for $x\in(0,1]$,
\begin{equation}\label{eq:intro-e1-TE}
d_1(x)\coloneqq\left\lfloor \frac{1}{x}\right\rfloor+1,
\qquad
T(x)\coloneqq d_1(x)x-1\in(0,1],
\end{equation}
and define, for $n\ge1$,
\begin{equation}\label{eq:intro-en}
d_{n+1}(x)=d_1\left(T^{n}(x)\right),
\end{equation}
where $T^n$ denotes the $n$th iterate of $T$. If $1/x\notin\mathbb N$, then $d_1(x)=\lceil 1/x\rceil$ and $0<T(x)<1$; if $x=1/m$ ($m\in\mathbb N$), then $d_1(x)=m+1$ and $T(1/m)=1/m$. Thus, $(d_n(x))_{n\ge1}$ is nondecreasing and eventually constant equal to $m+1$. It is a classical result that expansion \eqref{eq:intro-Engel} with digits given by \eqref{eq:intro-en} converges to $x$ (see, for example, \cite{Gal76}).

Two standard, coexisting conventions for rationals are common. One stops the iteration of $x\mapsto x\lceil1/x\rceil-1$ at $0$ and yields a finite expansion. If we adopt the conventions $c\cdot\infty\coloneqq\infty$ ($c>0$) and $1/\infty \coloneqq 0$, then appending an $\infty$-tail to this finite expansion does not affect the resulting sum in \eqref{eq:intro-Engel}. The other, which we used in \eqref{eq:intro-e1-TE}--\eqref{eq:intro-en}, always replaces $\lceil1/x\rceil$ by $\lfloor1/x\rfloor+1$; at $x=1/m$ it fixes $1/m$ and thus yields the constant tail $m{+}1,m{+}1,\dotsc$ automatically. Both conventions agree on irrationals and differ only at rationals. We adopt the nonterminating convention (without $\infty$-tails) throughout.

Let $\mathbb N_\infty\coloneqq\mathbb N\cup\{\infty\}$, and let $\Sigma_E$ denote the space of Engel digit sequences (nondecreasing sequences in $\mathbb N_\infty$); its precise definition, along with the metric topology, is given in Section \ref{subsec:symbolic-metric}. 
Let $\varphi\colon\Sigma_E\to[0,1]$ and $f\colon(0,1]\to\Sigma_E$ denote the evaluation and digit coding maps, respectively; precise definitions and basic properties are given in Section \ref{subsec:evaluation-coding}.

Fractal aspects of Engel expansions have been widely studied (see \cite{Wu00,LW01,LW03,She10,FW18,LL18,SW20,SW21,FS24}). Metric and number-theoretic topics trace back at least to Erd\H os--R\'enyi--Sz\H usz~\cite{ERS58}. Topological uniqueness on irrationals and the behaviour at rationals are classical \cite{Gal76}; Baire category results for certain level sets were analysed in \cite{SW21}. Nevertheless, to the best of our knowledge, a canonical topology on the space of Engel digit sequences, together with a systematic analysis of $\varphi$ and $f$ in that topology, appears to be missing.

We establish this framework by showing that the evaluation map $\varphi$ is globally Lipschitz continuous (Theorem \ref{thm:phi-cont}) and that the digit coding map $f$ is continuous precisely at the irrationals (Theorem \ref{thm:f-cont-irr}) and one-sided continuous at the rationals (Theorem \ref{thm:f-one-sided}). This structure simplifies the Baire category arguments for functions of the digits; in particular, it yields a concise proof that $\{x\in(0,1]:\lambda(x)=\infty\}$, where $\lambda\colon(0,1]\to[0,\infty]$ denotes the {\em convergence exponent} of Engel digit sequences, is comeagre in $(0,1]$, as previously established in \cite[Theorem~3.5]{SW21}. For comparison, an analogous approach for the Pierce expansion (also known as the alternating Engel expansion) appears in \cite{AhnJFG,AhnIJNT}.

This paper is organised as follows. Section~\ref{sec:preliminaries} introduces the symbolic space, the metric, and the associated maps, and recalls the classical fundamental intervals. Section~\ref{sec:main} establishes continuity properties of $\varphi$ and $f$. Section~\ref{sec:applications} applies the framework to show that broad divergence sets contain dense $G_\delta$ subsets (and are therefore comeagre), and concludes that $\{x\in(0,1]:\lambda(x)=\infty\}$ is comeagre.

\section{Preliminaries}\label{sec:preliminaries}

\subsection{Symbolic space and metric}\label{subsec:symbolic-metric}
Equip $\mathbb N$ with the discrete topology and set $X\coloneqq\mathbb N_\infty=\mathbb N\cup\{\infty\}$. Define
\[
\rho(x,y)\coloneqq
\begin{cases}
0,&x=y,\\
x^{-1}+y^{-1},&x\ne y,
\end{cases}
\qquad (\infty^{-1}\coloneqq0).
\]
By \cite[Lemma~3.1]{AhnJFG}, $\rho$ induces the one-point compactification topology on $X$. On $X^{\mathbb N}$, we use the product metric
\begin{equation}\label{eq:metric}
d(\alpha,\beta)\coloneqq\sum_{n=1}^\infty 2^{-n}\rho(\alpha_n,\beta_n),
\qquad\alpha=(\alpha_n)_n,\ \beta=(\beta_n)_n,
\end{equation}
which metrises the product topology and is topologically equivalent to the factorial-weight metric in \cite[Section 3.2]{AhnJFG}.

Define the \emph{space of Engel sequences}
\[
\Sigma_E\coloneqq\left\{\eta=(\eta_n)_{n\ge1}\in X^{\mathbb N}:2\le \eta_1\le \eta_2\le\dotsb\right\},
\]
with the subspace topology from $X^{\mathbb N}$. Put
\[
\Sigma_E^{\mathrm{irr}}\coloneqq\left\{\eta\in \Sigma_E: \eta_n \in \mathbb{N} \text{ for all $n$ and } \lim_{n\to\infty} \eta_n=\infty\right\},\qquad
\Sigma_E^{\mathrm{rat}}\coloneqq\Sigma_E\setminus\Sigma_E^{\mathrm{irr}}.
\]
Note that $\Sigma_E^{\mathrm{rat}}$ consists of sequences that are eventually constant in $\mathbb{N}$ or attain $\infty$ at some index.

\begin{lemma}\label{lem:SigmaE-closed}
The space of Engel sequences $\Sigma_E$ is closed in $X^{\mathbb N}$.
\end{lemma}

\begin{proof}
We show $X^{\mathbb N}\setminus\Sigma_E$ is open. Let $\beta\in X^{\mathbb N}\setminus\Sigma_E$. Then either $\beta_1=1$ or there exists $k\ge1$ with $\beta_{k}>\beta_{k+1}$.

If $\beta_1=1$, the cylinder $\{\alpha\in X^{\mathbb N}:\alpha_1=1\}$ is an open neighbourhood of $\beta$ disjoint from $\Sigma_E$ (since Engel sequences satisfy $\alpha_1\ge2$).

Otherwise pick $k$ with $\beta_{k}>\beta_{k+1}$. There are two subcases. (Note that $\beta_{k+1}\neq\infty$, since $\beta_k>\beta_{k+1}$ would be impossible.)

{\sc Case I: Both coordinates finite.} If $\beta_k,\beta_{k+1}\in\mathbb N$,
the cylinder
\[
U \coloneqq \{\alpha\in X^{\mathbb N}:\ \alpha_k=\beta_k,\ \alpha_{k+1}=\beta_{k+1}\}
\]
is open and every $\alpha\in U$ still violates monotonicity at $k$. Hence $U$ is disjoint from $\Sigma_E$.

{\sc Case II: Left coordinate $\infty$.} If $\beta_k=\infty$ and $\beta_{k+1}\in\mathbb N$, let
\[
U_k \coloneqq \{\infty\}\cup\{m\in\mathbb N: m\ge \beta_{k+1}+1\}
\quad\text{and}\quad
U_{k+1} \coloneqq \{\beta_{k+1}\}.
\]
Then $U_k$ is an open neighbourhood of $\infty$ in $\mathbb N_\infty$ (its complement is finite) and $U_{k+1}$ is open (singleton in the discrete part $\mathbb{N}$). The basic open set
\[
V \coloneqq \{\alpha\in X^{\mathbb N}: \alpha_k\in U_k, \, \alpha_{k+1}\in U_{k+1}\}
\]
contains $\beta$ and is disjoint from $\Sigma_E$, because for any $\alpha\in V$ one has either $\alpha_k=\infty>\alpha_{k+1}$ or $\alpha_k\ge \beta_{k+1}+1>\alpha_{k+1}$.

In all cases we have found an open neighbourhood of $\beta$ disjoint from $\Sigma_E$, so $X^{\mathbb N}\setminus\Sigma_E$ is open, that is, $\Sigma_E$ is closed.
\end{proof}

\begin{lemma}\label{lem:compact}
The product space $X^{\mathbb N}$ is compact; therefore $\Sigma_E$ is compact.
\end{lemma}

\begin{proof}
Since $X$ is compact (one-point compactification), Tychonoff's theorem implies that $X^{\mathbb N}$ is compact in the product topology, which is metrised by the metric $d$ in \eqref{eq:metric}. Then Lemma~\ref{lem:SigmaE-closed} yields compactness of $\Sigma_E$.
\end{proof}

\subsection{Cylinders and fundamental intervals}\label{subsec:fund-intervals}
A finite sequence $\sigma=(\sigma_1,\dotsc,\sigma_n)$ of integers with $2\le\sigma_1\le\cdots\le\sigma_n$ is called an \emph{admissible word of length $n$} (see \cite[Definition~2.1 and Proposition~2.2]{SW21}). We write $|\sigma|=n$ for the length of $\sigma$.

For an admissible word $\sigma$ of length $n$, define the \emph{level-$n$ cylinder}
\[
C_n(\sigma)\coloneqq\left\{\eta\in\Sigma_E: \eta_1=\sigma_1,\ \dotsc,\ \eta_n=\sigma_n\right\}.
\]
The corresponding \emph{level-$n$ fundamental interval} is the preimage under $f$:
\[
J_n(\sigma)\coloneqq f^{-1}\left(C_n(\sigma)\right)\subseteq(0,1].
\]

\begin{proposition}[{\cite[page~84]{Gal76}}]\label{prop:fundamental}
For an admissible word $\sigma$ of length $n$ as above,
\[
J_n(\sigma)=\left(\ell(\sigma),\,r(\sigma)\right],
\]
where
\[
\ell(\sigma) \coloneqq \sum_{k=1}^n\frac{1}{\sigma_1\dotsm\sigma_k},\quad
r(\sigma) \coloneqq \sum_{k=1}^{n-1}\frac{1}{\sigma_1\dotsm\sigma_k}
+\frac{1}{\sigma_1\dotsm\sigma_{n-1}(\sigma_n-1)}.
\]
In particular, $C_n(\sigma)$ is clopen in $\Sigma_E$ and $\varphi\left(C_n(\sigma)\right)=[\ell(\sigma),r(\sigma)]$.
\end{proposition}

\begin{remark}
The image $\varphi\left(C_n(\sigma)\right)$ attains its left endpoint $\ell(\sigma)$ by appending $\sigma$ with an $\infty$-tail, $(\sigma_1,\dotsc,\sigma_n,\infty,\infty,\dotsc)$, and its right endpoint $r(\sigma)$ by the nonterminating sequence in the cylinder that maximises $\varphi$, $(\sigma_1,\dotsc,\sigma_{n-1},\sigma_n,\sigma_n,\dotsc)$ (which is the nonterminating code for $r(\sigma)$).
\end{remark}

\subsection{Evaluation and coding}\label{subsec:evaluation-coding}

The evaluation map $\varphi\colon\Sigma_E\to[0,1]$ is defined by
\[
\varphi(\eta)\coloneqq\sum_{n=1}^\infty\frac{1}{\eta_1\eta_2\dotsm\eta_n}.
\]
Since $(\eta_n)_{n\ge 1}$ is nondecreasing with $\eta_n\ge2$, the tail is dominated by a geometric series, so $\varphi$ is well defined on $\Sigma_E$.
(We adopt the conventions $\infty^{-1}\coloneqq0$ and, for products, that if $\eta_k=\infty$ then $(\eta_1\dotsm\eta_n)^{-1}=0$ for all $n\ge k$; hence the tail vanishes once a coordinate equals $\infty$.)

The digit coding map $f\colon(0,1]\to\Sigma_E$ is given by $f(x)\coloneqq(d_n(x))_{n\ge1}$, where $d_1,\,T,\,d_{n+1}$ are as in \eqref{eq:intro-e1-TE}--\eqref{eq:intro-en}. Under our nonterminating convention, rationals are encoded by eventually constant integer sequences; for instance, $f(3/4)=(2,3,3,3,\dotsc)$. Note that $\varphi(2,3,3,\dotsc)=3/4=\varphi(2,2,\infty,\infty,\dotsc)$, so $\varphi$ is not injective on $\Sigma_E$ at rationals.

\begin{proposition}[\cite{Gal76}]\label{prop:basic-inversion}
We have $\varphi\circ f=\id_{(0,1]}$, and $f\circ\varphi=\id$ on $\Sigma_E^{\mathrm{irr}}$.
\end{proposition}

\begin{proof}
By construction, the Engel expansion of $x\in(0,1]$ is unique, so $\varphi(f(x))=x$.
If $\eta\in\Sigma_E^{\mathrm{irr}}$, then $\eta_n\in\mathbb N$ with $\eta_n\to\infty$, and $\varphi(\eta)$ has Engel digits $\eta_n$; hence $f(\varphi(\eta))=\eta$.
\end{proof}

\section{Main results: continuity and symbolic structure}\label{sec:main}

\subsection{Continuity of the evaluation map}

\begin{theorem}\label{thm:phi-cont}
The evaluation map $\varphi\colon\Sigma_E\to[0,1]$ is Lipschitz (for $d$ in \eqref{eq:metric} and the Euclidean metric).
\end{theorem}

\begin{proof}
Let $\eta,\eta'\in\Sigma_E$ and let $n$ be the first index at which they differ. Put $A\coloneqq1/(\eta_1\cdots\eta_{n-1})=1/(\eta_1'\cdots\eta_{n-1}')$. Then
\[
\varphi(\eta)-\varphi(\eta')
= A\left(\sum_{j=0}^{\infty}\frac{1}{\eta_n\dotsm \eta_{n+j}}
-\sum_{j=0}^{\infty}\frac{1}{\eta_n'\dotsm \eta_{n+j}'}\right).
\]
Since the digits are nondecreasing, $\eta_{n+\ell}\ge \eta_n$ and $\eta_{n+\ell}'\ge\eta_n'$ for all $\ell\ge0$. Hence,
\[
\sum_{j=0}^\infty \frac{1}{\eta_n\dotsm\eta_{n+j}} \leq \sum_{j=0}^\infty \frac{1}{\eta_n^{j+1}}=\frac{1}{\eta_n-1}\le\frac{2}{\eta_n},
\]
and similarly with $\eta$ replaced by $\eta'$. Therefore
\[
\left|\varphi(\eta)-\varphi(\eta')\right|
\le A\left(\frac{2}{\eta_n}+\frac{2}{\eta_n'}\right)
\le 2^{2-n}\left(\frac{1}{\eta_n}+\frac{1}{\eta_n'}\right),
\]
since $\eta_k,\eta_k'\ge2$ imply $A\le 2^{-(n-1)}$. On the other hand,
\[
d(\eta,\eta')\ge 2^{-n}\rho(\eta_n,\eta_n')=2^{-n}\left(\frac{1}{\eta_n}+\frac{1}{\eta_n'}\right).
\]
Combining gives $\left|\varphi(\eta)-\varphi(\eta')\right|\le 4\,d(\eta,\eta')$.
\end{proof}

\subsection{Continuity properties of the coding map}\label{subsec:coding-cont}

Throughout this subsection we work in the compact metric space $(\Sigma_E,d)$ (Lemma~\ref{lem:compact}) and use that $\varphi$ is Lipschitz (Theorem~\ref{thm:phi-cont}). We also recall Proposition~\ref{prop:basic-inversion}: $\varphi\circ f=\id_{(0,1]}$ and $f\circ\varphi=\id$ on $\Sigma_E^{\mathrm{irr}}$.

\begin{lemma}\label{lem:two-codes}
Let $x\in(0,1)\cap\mathbb Q$. Then there is a unique admissible word $\sigma=(\sigma_1,\dots,\sigma_n)$ such that
\[
x=r(\sigma)=\sum_{k=1}^{n-1}\frac{1}{\sigma_1\cdots\sigma_k}
+\frac{1}{\sigma_1\cdots\sigma_{n-1}(\sigma_n-1)}.
\]
Under the nonterminating convention,
\[
\alpha\coloneqq f(x)=(\sigma_1,\dots,\sigma_{n-1},\sigma_n,\sigma_n,\sigma_n,\dotsc).
\]
The terminating Engel code of the same $x$ is obtained by decreasing the last eligible digit by $1$ and appending an $\infty$-tail: there exists an index
\[
j\coloneqq\max\{1\le k\le n: \sigma_k\ge 3\}
\]
such that
\[
\alpha'\coloneqq(\sigma_1,\dots,\sigma_{j-1},\sigma_j-1,\infty,\infty,\dotsc)\in\Sigma_E
\quad\text{and}\quad
\varphi(\alpha')=\varphi(\alpha)=x.
\]
In particular, $\varphi^{-1}(\{x\})=\{\alpha,\alpha'\}$ and $f(x)=\alpha$.
\end{lemma}

\begin{proof}
The description of $J_n(\sigma)=(\ell(\sigma),r(\sigma)]$ in Proposition~\ref{prop:fundamental} shows that $x$ is the (right) endpoint of a unique fundamental interval, giving the stated $\sigma$ and $\alpha$.
For the terminating code, one checks directly that replacing the rightmost digit greater than or equal to $3$ by $1$ less and appending an $\infty$-tail preserves admissibility and the value under $\varphi$ (the usual carry rule); see also \cite{Gal76}. Uniqueness of the two codes follows from monotonicity and the endpoint structure of fundamental intervals.
\end{proof}

\begin{theorem}\label{thm:f-cont-irr}
The map $f:(0,1]\to\Sigma_E$ is continuous at every irrational $x$.
\end{theorem}

\begin{proof}
Suppose $x\in(0,1]\setminus\mathbb Q$ and $f$ is not continuous at $x$.
Then there are $\varepsilon>0$ and $x_m\to x$ with $d\left(f(x_m),f(x)\right)\ge\varepsilon$ for all $m$.
Since $(\Sigma_E,d)$ is compact (Lemma \ref{lem:compact}), the sequence $(f(x_m))_{m\ge1}$ has a convergent subsequence, say $f(x_{m_k})\to\beta\in\Sigma_E$.
Since $\varphi$ is continuous (Theorem~\ref{thm:phi-cont}),
\[
x=\lim_{k\to\infty} x_{m_k}=\lim_{k\to\infty} \varphi\left(f(x_{m_k})\right)=\varphi(\beta).
\]
As $x$ is irrational, $\beta\in\Sigma_E^{\mathrm{irr}}$, and hence,
$\beta=f(\varphi(\beta))=f(x)$ (Proposition~\ref{prop:basic-inversion}).
Thus $f(x_{m_k})\to f(x)$, contradicting $d(f(x_{m_k}),f(x))\ge\varepsilon$.
\end{proof}

\begin{theorem}\label{thm:f-one-sided}
Let \(x\in(0,1]\cap\mathbb Q\). Then \(f\) is left-continuous at \(x\).
Moreover, if \(x\in(0,1)\cap\mathbb Q\), then \(f\) is not right-continuous at \(x\).
\end{theorem}

\begin{proof}
If $x=1$, the proof is similar to that of Theorem~\ref{thm:f-cont-irr}, using the fact that
\[
\varphi^{-1}(\{1\})=\{(2,2,2,\dotsc)\}.
\]
Thus $f$ is left-continuous at $1$. We may therefore assume that
$x\in(0,1)\cap\mathbb Q$.
Let $\sigma=(\sigma_1,\dots,\sigma_n)$ be as in Lemma~\ref{lem:two-codes}, so
$x=r(\sigma)$ and $J_n(\sigma)=(\ell(\sigma),r(\sigma)]$ with right endpoint $x$.
Write $\alpha\coloneqq f(x)=(\sigma_1,\dots,\sigma_{n-1},\sigma_n,\sigma_n,\dotsc)$ and
$\alpha'\coloneqq (\sigma_1,\dots,\sigma_{j-1},\sigma_j-1,\infty,\infty,\dotsc)$ for the terminating code (Lemma~\ref{lem:two-codes}).
Recall from Proposition~\ref{prop:fundamental} that the cylinder
\[
C_n(\sigma)\coloneqq\{\eta\in\Sigma_E:\ \eta_1=\sigma_1,\dots,\eta_n=\sigma_n\}
\]
is clopen in $\Sigma_E$.

{\sc Left-continuity.}
Take any sequence $y_m\uparrow x$ with $y_m<x$.
For $m$ large, $y_m\in J_n(\sigma)$, hence $f(y_m)\in C_n(\sigma)$.
Because $C_n(\sigma)$ is closed, any subsequential limit $\gamma$ of $(f(y_m))_{m \geq 1}$ lies in $C_n(\sigma)$.
On the other hand, $\varphi$ is continuous (Theorem~\ref{thm:phi-cont}), so $\varphi(\gamma)=\lim_m \varphi(f(y_m))=\lim_m y_m=x$.
Thus, $\gamma\in\varphi^{-1}(\{x\})=\{\alpha,\alpha'\}$ (Lemma~\ref{lem:two-codes}).
Since $\alpha\in C_n(\sigma)$ but $\alpha'\notin C_n(\sigma)$ (as its first $n$ digits do not match $\sigma$), we must have $\gamma=\alpha=f(x)$.
Therefore, $f(y_m)\to f(x)$ as $m\to\infty$.

{\sc Failure of right-continuity.}
Take any sequence $z_m\downarrow x$ with $z_m>x$.
For each $m$, $z_m\notin J_n(\sigma)$, hence $f(z_m)\notin C_n(\sigma)$.
Since $\Sigma_E\setminus C_n(\sigma)$ is also closed (as $C_n(\sigma)$ is clopen), any subsequential limit $\delta$ of $(f(z_m))_{m \geq 1}$ must lie in $\Sigma_E\setminus C_n(\sigma)$.
Again by continuity of $\varphi$,
$\varphi(\delta)=\lim_m \varphi(f(z_m))=\lim_m z_m=x$,
so $\delta\in\{\alpha,\alpha'\}$.
But $\alpha\in C_n(\sigma)$ while $\alpha'\in \Sigma_E\setminus C_n(\sigma)$, hence $\delta=\alpha'\neq f(x)$.
Therefore $\lim_{t\downarrow x} f(t)=\alpha'\ne f(x)$, and $f$ is not right-continuous at $x$.
\end{proof}

\begin{corollary}\label{cor:homeo}
The restriction $\varphi:\Sigma_E^{\mathrm{irr}}\to(0,1]\setminus\mathbb Q$ is a homeomorphism with inverse $f$.
\end{corollary}

\begin{proof}
The bijectivity of $\varphi \colon \Sigma^{\mathrm{irr}}_E \to (0,1] \setminus \mathbb{Q}$ with inverse $f$ is established by Proposition~\ref{prop:basic-inversion}. The continuity of $\varphi$ is a consequence of Theorem~\ref{thm:phi-cont} (Lipschitz continuity). The continuity of the inverse map $f$ on $(0,1]\setminus\mathbb{Q}$ is established by Theorem~\ref{thm:f-cont-irr}. Since $\varphi$ is a continuous bijection with a continuous inverse, it is a homeomorphism.
\end{proof}

\section{Applications: comeagre divergence sets and further remarks}\label{sec:applications}

In \cite{AhnIJNT} we studied Pierce digits; here we use the present Engel topology to obtain a general Baire category statement for Engel digits.

\subsection{Comeagre divergence for functions of the digits}\label{subsec:comeagre-divergence}
For a nonnegative function $\phi\colon\{2,3,\dotsc\}\to[0,\infty)$, write
\[
\mathcal D_\phi\coloneqq\left\{x\in(0,1]: \sum_{n=1}^\infty \phi(d_n(x))=\infty\right\}.
\]

\begin{theorem}\label{thm:Gdelta}
Let $\phi\colon\{2,3,\dotsc\}\to[0,\infty)$ and assume that the set $\{t\ge2: \phi(t)>0\}$ is unbounded. Then there exists a dense $G_\delta$ set $\mathcal G_\phi\subseteq(0,1]$ with $\mathcal G_\phi\subseteq \mathcal D_\phi$. In particular, $\mathcal D_\phi$ is comeagre (residual) in $(0,1]$.
\end{theorem}

\begin{proof}
For $m\in\mathbb N$, define
\[
\Upsilon_m\coloneqq
\bigcup_{\substack{\sigma=(\sigma_1,\dotsc,\sigma_n) \text{ admissible}\\ \sum_{k=1}^n \phi(\sigma_k)>m}}
C_n(\sigma) \subseteq \Sigma_E.
\]
Each $\Upsilon_m$ is open (union of clopen cylinders). We claim that $\Upsilon_m\cap\Sigma_E^{\mathrm{irr}}$ is dense in $\Sigma_E^{\mathrm{irr}}$.
Indeed, given a nonempty cylinder $C_N(\tau)$ with $\tau_N\in\mathbb N$ (equivalently, $C_N(\tau)\cap\Sigma_E^{\mathrm{irr}}\neq\varnothing$), choose
$t\ge\max\{2,\tau_N\}$ with $\phi(t)>0$ (possible since $\{t: \phi(t)>0\}$ is unbounded), and append a long enough $t$-tail to $\tau$ so that $\sum_{k=1}^{|\sigma|}\phi(\sigma_k)>m$. Then
$C_{|\sigma|}(\sigma)\subseteq C_N(\tau)\cap\Upsilon_m$, proving density in $\Sigma_E^{\mathrm{irr}}$.

Hence,
\[
\Gamma_\phi\coloneqq\bigcap_{m=1}^\infty \Upsilon_m
\]
is such that $\Gamma_\phi\cap\Sigma_E^{\mathrm{irr}}$ is a dense $G_\delta$ in $\Sigma_E^{\mathrm{irr}}$; moreover, if $\eta\in\Gamma_\phi\cap\Sigma_E^{\mathrm{irr}}$, then the partial sums $\sum_{k=1}^n\phi(\eta_k)$ are unbounded in $n$, so $\sum_{k\ge1}\phi(\eta_k)=\infty$.

Set
\[
\mathcal G_\phi\coloneqq \varphi\left(\Gamma_\phi\cap\Sigma_E^{\mathrm{irr}}\right) \subseteq (0,1]\setminus\mathbb Q.
\]
By Corollary~\ref{cor:homeo}, $\varphi:\Sigma_E^{\mathrm{irr}}\to(0,1]\setminus\mathbb Q$ is a homeomorphism; hence $\mathcal G_\phi$ is a dense $G_\delta$ in $(0,1]\setminus\mathbb Q$, and for $x\in\mathcal G_\phi$ with $\eta\coloneqq f(x)$,
\[
\sum_{n=1}^\infty\phi\left(d_n(x)\right)=\sum_{n=1}^\infty\phi(\eta_n)=\infty.
\]
Since $(0,1]\setminus\mathbb Q$ is dense $G_\delta$ in $(0,1]$, it follows that $\mathcal G_\phi$ is a dense $G_\delta$ set in $(0,1]$. As $\mathcal G_\phi \subseteq \mathcal D_\phi$, we conclude that $\mathcal D_\phi$ is comeagre (residual) in $(0,1]$.
\end{proof}

\begin{corollary}\label{cor:Engel-ls}
For $s>0$, put
\[
\mathcal D (s)\coloneqq\left\{x\in(0,1]:\ \sum_{n=1}^{\infty}d_n(x)^{-s}=\infty\right\}.
\]
Then $\mathcal D (s)$ contains a dense $G_\delta$ and is comeagre in $(0,1]$.
\end{corollary}

\begin{proof}
Apply Theorem~\ref{thm:Gdelta} with $\phi(t)\coloneqq t^{-s}$.
\end{proof}

Recall the \emph{convergence exponent} of Engel digit sequences $\lambda:(0,1]\to[0,\infty]$ defined by
\[
\lambda(x)\coloneqq \inf\left\{s\ge0: \sum_{n=1}^\infty d_n(x)^{-s}<\infty\right\}.
\]
For irrational $x$, it is classical that $\lambda(x)=\limsup_{n\to\infty}\frac{\log n}{\log d_n(x)}$ (see \cite{SW21}).

\begin{corollary}[{\cite[Theorem~3.5]{SW21}}]\label{cor:lambda-comeagre}
The set $\{x\in(0,1]:\ \lambda(x)=\infty\}$ is comeagre in $(0,1]$.
Equivalently, for every $\alpha\in[0,\infty)$, the set $\{x\in(0,1]:\ \lambda(x)\le \alpha\}$ is meagre in $(0,1]$.
\end{corollary}

\begin{proof}
For each $k\in\mathbb N$, Corollary~\ref{cor:Engel-ls} shows that $\mathcal D (k)$ is comeagre. Since
\[
\{x\in(0,1]:\lambda(x)=\infty\}=\bigcap_{k=1}^\infty\mathcal D (k),
\]
this set is a countable intersection of comeagre sets and is therefore comeagre. For the equivalence, note that if $\lambda(x) > \alpha$, there exists a rational $q > \alpha$ such that $\sum_{n=1}^\infty d_n(x)^{-q} = \infty$. Thus,
\[
\{x\in(0,1]:\lambda(x) > \alpha\} = \bigcup_{q\in\mathbb Q,\ q>\alpha} \mathcal D(q).
\]
As a countable union of comeagre sets, this set is comeagre, which implies that its complement $\{x\in(0,1]:\lambda(x)\le\alpha\}$ is meagre.
\end{proof}

\subsection{Concluding remarks}
The symbolic topology for Engel expansions yields a clean qualitative picture: $\varphi$ is globally Lipschitz, while $f$ is continuous exactly at irrationals and one-sided continuous at rationals. The comeagre divergence phenomena in Theorem~\ref{thm:Gdelta} and Corollary~\ref{cor:Engel-ls} follow from the clopen cylinder structure on $\Sigma_E$ together with the homeomorphism $\varphi \colon \Sigma_E^{\mathrm{irr}}\to (0,1]\setminus\mathbb{Q}$; the fundamental intervals in $(0,1]$ are simply the mirror of this structure under $\varphi$. In short, the canonical Engel topology renders the continuity and Baire category results as immediate consequences of the symbolic structure.

\end{document}